\documentclass[12pt, reqno]{amsart}
\usepackage{amsfonts}
\usepackage{amsmath, amsthm, amscd, amsfonts, amssymb, graphicx, color, enumerate}
\usepackage[bookmarksnumbered, colorlinks, plainpages]{hyperref}

\newtheorem{theorem}{Theorem}[section]
\newtheorem{lemma}[theorem]{Lemma}

\newtheorem{corollary}[theorem]{Corollary}
\theoremstyle{definition}

\theoremstyle{remark}
\newtheorem{remark}[theorem]{Remark}
\numberwithin{equation}{section}

\begin{document}
\setcounter{page}{1}

\title{A submetric characterization of Rolewicz's property ($\beta$)}

\author{Sheng Zhang}

\address{School of Mathematics, Southwest Jiaotong University, Chengdu, Sichuan 611756, China}
\email{\textcolor[rgb]{0.00,0.00,0.84}{sheng@swjtu.edu.cn}}

%\subjclass[2010]{Primary 46B80, 46B06.}

\thanks{The author was supported by National Natural Science Foundation of China, grant numbers 11801469 and 12071389}

\begin{abstract}
The main result is a submetric characterization of the class of Banach spaces admitting an equivalent norm with Rolewicz's property ($\beta$). As applications we prove that up to renorming, property ($\beta$) is stable under coarse Lipschitz embeddings and coarse quotients.
\end{abstract}

\maketitle

\section{Introduction}

Metric characterization of Banach space properties is one of the most intriguing problems in nonlinear geometry of Banach spaces. Being a part of the Ribe program, it aim to seek pure metric interpretation of important concepts in Banach space theory. The first step in this direction is Bourgain's characterization of superreflexivity in terms of Lipschitz embeddability of binary trees \cite{Bourgain1986}. Metric characterizations of other properties such as RNP, reflexivity, Rademacher's type and cotype have been discovered thereafter. We refer to the survey paper \cite{Ostrovskii2016} and references therein for a list of Banach space properties that are known to have such characterization.

The main concern of this paper is Rolewicz's property ($\beta$), which is an asymptotic property of Banach spaces introduced by Rolewicz \cite{Rolewicz1987}. We recall the equivalent definition from \cite{Kutzarova1991} that a Banach space $X$ is said to have property ($\beta$) if for every $t\in(0,a]$, where $a\in[1,2]$ depends only on the space $X$, there exists $\delta=\delta(t)\in(0,1)$ such that whenever $x,x_n\in B_X$ with $\inf_{i\neq j}\|x_i-x_j\|\ge t$, there exists $k\in\mathbb{N}$ such that $$\frac{\|x-x_k\|}{2}\le1-\delta.$$
A modulus for the property ($\beta$) is defined by $$\bar{\beta}_X(t)=1-\sup\left\{\inf_{n\ge1}\left\{\frac{\|x-x_n\|}{2}\right\}:~x,x_n\in B_X,~\inf\limits_{i\neq j}\|x_i-x_j\|\ge t\right\}.$$
Then $X$ has property ($\beta$) if and only if $\bar{\beta}_X(t)>0$ for all $t>0$.

Baudier, Kalton and Lancien proved in \cite{BKL2010} that within the class of reflexive Banach spaces, the class of Banach spaces admitting an equivalent norm with property ($\beta$) can be characterized in terms of Lipschitz embeddability of the countably branching tree. In this paper, our goal is to provide a different metric characterization of property ($\beta$). The characterization we obtained, in the language of Ostrovskii \cite{Ostrovskii2014}, is called a submetric characterization; it combines ideas from \cite{BZ2016,DKR2016} and a self-improvement argument first appeared in \cite{JS2009}.

We gather some necessary definitions and preliminary results in section 2. Section 3 is devoted to the submetric characterization of property ($\beta$). The last section contains two applications regarding the stability of property ($\beta$) under coarse Lipschitz embeddings and nonlinear quotients. We use standard notations: $B_X(x,r)$ denotes the closed ball in a metric space $X$ with center $x$ and radius $r$. If $X$ is a Banach space, $B_X:=B_X(0,1)$ is the closed unit ball of $X$. All Banach spaces are assumed to be real.

\section{Preliminaries}

\subsection{Uniform and coarse continuity}\ 

Let $f:X\to Y$ be a map between metric spaces $X$ and $Y$. The modulus of continuity of $f$ is defined by $$\omega_f(t):=\sup\{d_Y(f(x),f(y)): d_X(x,y)\le t\}.$$ 
The map $f$ is said to be uniformly continuous if $\lim_{t\to0}\omega_f(t)=0$; $f$ is said to be coarsely continuous if  $\omega_f(t)<\infty$ for all $t>0$. The Lipschitz constant of $f$ is defined by $$\text{Lip}(f):=\sup\left\{\frac{d_Y(f(x),f(y))}{d_X(x,y)}: x,y\in X, x\neq y\right\},$$ and $f$ is called a Lipschitz map if $\text{Lip}(f)<\infty$. 
We say $f$ is a Lipschitz embedding if there exists constants $C,L>0$ such that for all $x,y\in X$,
$$\frac{1}{C}d_X(x,y)\le d_Y(f(x),f(y))\le Ld_X(x,y).$$

If $X$ is unbounded, we define $$\text{Lip}_d(f):=\sup\left\{\frac{d_Y(f(x),f(y))}{d_X(x,y)}: x,y\in X, d_X(x,y)\ge d\right\},$$
and say that $f$ is Lipschitz for large distances if $\text{Lip}_d(f)<\infty$ for all $d>0$. If there exists $d>0$ such that $\text{Lip}_d(f)<\infty$, then we say $f$ is coarse Lipschitz. $f$ is said to be a coarse Lipschitz embedding if there exists constants $d\ge0$ and $A,B>0$ such that for all $x,y\in X$ with $d_X(x,y)\ge d$,
$$\frac{1}{A}d_X(x,y)\le d_Y(f(x),f(y))\le Bd_X(x,y).$$

A metric space $X$ is said to be metrically convex if for every $x_0, x_1\in X$ and $0<\lambda<1$, there exists $x_\lambda\in X$ such that $d(x_0,x_\lambda)=\lambda d(x_0, x_1)$ and $d(x_1, x_\lambda)=(1-\lambda)d(x_0, x_1)$. Uniformly continuous maps and coarsely continuous maps defined on metrically convex metric spaces must be Lipschitz for large distances.

\subsection{The countably branching tree}\ 

The countably branching tree of infinite height, denoted by $\mathcal{T}$, is an unweighted rooted tree each of whose vertices has countably infinite edges incident to it. We still use $\mathcal{T}$ to denote its vertex set, then $\mathcal{T}$ has a one-to-one correspondence with all finite subsets of $\mathbb{N}$, where $\emptyset$ is the root of $\mathcal{T}$, and every other vertices can be represented as $(n_1<n_2<...<n_k)$ for some $k\in\mathbb{N}$.
The ancestor-to-descendant relations between vertices is defined as follows: the root $\emptyset$ is the ancestor of all the other vertices; a vertex $I=(m_1<m_2<...<m_l)$ is said to be an ancestor of a vertex $J=(n_1<n_2<...<n_k)$ (or $J$ is a descendant of $I$), written as $I<J$, if $l<k$ and $m_j=n_j$ for all $1\le j\le l$. If $I<J$ and $k=l+1$ then we say $J$ is an immediate descendant of $I$.
The height of a vertex $J$, denoted by $|J|$, is the number of edges jointing $J$ and the root, which is also the number of elements in $J$ viewed as a finite subset of $\mathbb{N}$; here we set $|\emptyset|=0$ by convention. The notation $\max J$ and $\min J$ denote the largest and the smallest natural number in $J$ respectively.
$\mathcal{T}$ will always be equipped with the shortest path metric, i.e., the distance between two vertices $I,J\in\mathcal{T}$ is defined by $$d_{\mathcal{T}}(I,J):=|I|+|J|-2|\text{gca}(I,J)|,$$ where $\text{gca}(I,J)$ is the common ancestor of $I$ and $J$ with the greatest height.

$\mathcal{T}$ has an important property that it is self-contained, meaning that $\mathcal{T}$ contains proper subsets that are rescaled isometric copies of itself. In the lemma below we choose a sequence of such subsets in a way that vertices of the same height are ``separated''.

\begin{lemma}\label{subtree}
There exists a sequence $\{\mathcal{T}_n\}_{n=0}^\infty$ of subsets of $\mathcal{T}$ that satisfies the following:
\begin{enumerate}
\item $\mathcal{T}_n\supseteq\mathcal{T}_{n+1}$ for every $n\in\mathbb{N}\cup\{0\}$.
\item For every $n\in\mathbb{N}\cup\{0\}$, $\mathcal{T}_n$ is $2^n$-isometric to $\mathcal{T}$, i.e., there exists a surjective map $i_n:\mathcal{T}\to\mathcal{T}_n$ such that $d_{\mathcal{T}}(i_n(I),i_n(J))=2^n d_{\mathcal{T}}(I,J)$ for all $I,J\in\mathcal{T}$. Therefore, $\mathcal{T}_n$ can be viewed as a weighted countably branching tree each of whose edges has weight $2^n$.
\item For every $n\in\mathbb{N}\cup\{0\}$ and every vertex $J\in\mathcal{T}_n$, the immediate descendants of $J$ in $\mathcal{T}_n$, denoted by $\{J\cup J_i\}_{i=1}^\infty$ where $|J_i|=2^n$ and $\max J<\min J_i$ for all $i$, satisfies $\max J_i<\min J_{i+1}$ for all $i$.
\end{enumerate}
\end{lemma}

\begin{proof}
Set $\mathcal{T}_0=\mathcal{T}$. Suppose $\mathcal{T}_n$ has been defined that satisfies (1)--(3), we define $\mathcal{T}_{n+1}$ as follows: first choose the root $\emptyset$. Suppose $J\in\mathcal{T}_n$ has been chosen. Let $\{J\cup J_i\}_{i=1}^\infty$ be the immediate descendants of $J$ in $\mathcal{T}_n$, where $|J_i|=2^n$, $\max J<\min J_i$ and $\max J_i<\min J_{i+1}$ for all $i$. Choose an arbitrary immediate descendant of $J\cup J_1$ in $\mathcal{T}_n$, denoted by $J\cup \bar{J}_1$, where $|\bar{J}_1|=2^{n+1}$ and $J_1<\bar{J}_1$. Set $k_1=1$ and suppose $J\cup\bar{J}_{k_j}$ has been chosen, since $\max J_i<\min J_{i+1}$ for all $i$, there exists $k_{j+1}\in\mathbb{N}$ such that $\min J_{k_{j+1}}>\max \bar{J}_{k_j}$. Choose an arbitrary immediate descendant of $J\cup J_{k_{j+1}}$ in $\mathcal{T}_n$, denoted by $J\cup \bar{J}_{k_{j+1}}$, where $|\bar{J}_{k_{j+1}}|=2^{n+1}$ and $J_{k_{j+1}}<\bar{J}_{k_{j+1}}$. Then $\{J\cup\bar{J}_{k_j}\}_{j=1}^\infty$ is a sequence of descendants of $J$ with height $|J|+2^{n+1}$ and satisfies $\max \bar{J}_{k_j}<\min\bar{J}_{k_{j+1}}$ for all $j$. Now the vertex set of $\mathcal{T}_{n+1}$ are those vertices chosen in this induction process. It is easy and left to the reader to check that $\{\mathcal{T}_n\}_{n=0}^\infty$ is a sequence that satisfies (1)--(3).
\end{proof}

We call such $\mathcal{T}_n$'s pruned isometric subsets of $\mathcal{T}$, where the subscript $n$ means $\mathcal{T}_n$ is $2^n$-isometric to $\mathcal{T}$ in the sense of (2), and $\mathcal{T}_0=\mathcal{T}$. In Lemma \ref{subtree} the sequence $\{\mathcal{T}_n\}_{n=0}^\infty$ are built from $\mathcal{T}$, but in the sequel we will need to build such a sequence starting from a fix pruned isometric subset $\mathcal{T}_k$, i.e., there is a sequence of subsets of $\mathcal{T}_k$, denoted by $\{\mathcal{T}_{k+n}\}_{n=0}^\infty$, which satisfies the following:
\begin{enumerate}[(1')]
\item $\mathcal{T}_{k+n}\supseteq\mathcal{T}_{k+n+1}$ for every $n\in\mathbb{N}\cup\{0\}$.
\item For every $n\in\mathbb{N}\cup\{0\}$, $\mathcal{T}_{k+n}$ is $2^{k+n}$-isometric to $\mathcal{T}$.
\item For every $n\in\mathbb{N}\cup\{0\}$ and every vertex $J\in\mathcal{T}_{k+n}$, the immediate descendants of $J$ in $\mathcal{T}_{k+n}$, denoted by $\{J\cup J_i\}_{i=1}^\infty$ where $|J_i|=2^{k+n}$ and $\max J<\min J_i$ for all $i$, satisfies $\max J_i<\min J_{i+1}$ for all $i$. 
\end{enumerate}

\medskip
$\mathcal{T}$ also plays a significant role in the metric characterization of Banach space properties. As mentioned in the introduction, it was proved in \cite{BKL2010} that within the class of reflexive Banach spaces, Rolewicz's property ($\beta$) can be characterized by Lipschitz embeddability of $\mathcal{T}$.

\begin{theorem}[\cite{BKL2010}]\label{betatree}
A reflexive Banach space $X$ DOES NOT have an equivalent norm with property ($\beta$) if and only $\mathcal{T}$ admits a Lipschitz embedding into $X$.
\end{theorem}

\section{A submetric characterization of property ($\beta$)}

The following is a submetric characterization of the class of Banach spaces admitting an equivalent norm with Rolewicz's property ($\beta$):

\begin{theorem}\label{mcbeta}
Let $X$ be a Banach space. The following assertions are equivalent:
\begin{enumerate}[(i)]
\item $X$ DOES NOT have an equivalent norm with property ($\beta$).
\item There exist a pruned isometric subset $\mathcal{T}_k$ of $\mathcal{T}$, $k\in\mathbb{N}\cup\{0\}$, a Lipschitz map $\phi:\mathcal{T}_k\to X$ and a constant $\gamma>0$, such that
for every $J\in\mathcal{T}_k$, and every pair of descendants $\bar{J}_1=J\cup J_1$, $\bar{J}_2=J\cup J_2$ of $J$ in $\mathcal{T}_k$ that satisfies $|J_1|=|J_2|$ and $\max J_1<\min J_2$, $$\|\phi(\bar{J}_1)-\phi(\bar{J}_2)\|\ge\gamma d_{\mathcal{T}}(\bar{J}_1,\bar{J}_2).$$
\end{enumerate}
\end{theorem}

\begin{proof}
(i)$\Rightarrow$(ii): If $X$ is reflexive but does not have an equivalent norm with property ($\beta$), then it follows from Theorem \ref{betatree} that $\mathcal{T}$ admits a Lipschitz embedding into $X$. Note that the Lipschitz embedding satisfies (ii) for $k=0$, and we can assume that it has Lipschitz constant 1 so that $\gamma\in(0,1]$.

If $X$ is non-reflexive, fix $\theta\in(0,1)$, then it follows from James's characterization of reflexivity \cite{James} that there exists $\{x_j\}_{j=1}^\infty\subseteq B_X$ and $\{x^*_j\}_{j=1}^\infty\subseteq B_{X^*}$ such that $x^*_m(x_n)=\theta$ if $m\le n$ and $x^*_m(x_n)=0$ if $m>n$. Define $\phi:\mathcal{T}\to X$ by $\phi(J)=\sum_{j\in J}x_j$, $J\in\mathcal{T}$. Next we show that $\phi$ has Lipschitz constant at most 1 and satisfies (ii) for $k=0$ and $\gamma=\theta/2$:

If $J,\bar{J}\in\mathcal{T}$ and $J<\bar{J}$, let $l:=\min (\bar{J}\setminus J)$, then we have $$\|\phi(J)-\phi(\bar{J})\|=\Big\|\sum_{j\in\bar{J}\setminus J}x_j\Big\|\le |\bar{J}|-|J|=d_{\mathcal{T}}(J,\bar{J}),$$ and $$\|\phi(J)-\phi(\bar{J})\|=\Big\|\sum_{j\in\bar{J}\setminus J}x_j\Big\|\ge x^*_l\Big(\sum_{j\in\bar{J}\setminus J}x_j\Big)=(|\bar{J}|-|J|)\theta=\theta d_{\mathcal{T}}(J,\bar{J}).$$

Let $J\in\mathcal{T}_k$, $\bar{J}_1=J\cup J_1$ and $\bar{J}_2=J\cup J_2$ are descendants of $J$ in $\mathcal{T}_k$ that satisfies $|J_1|=|J_2|$ and $\max J_1<\min J_2$, let $k:=\min J_2$, then 
\begin{align*}
\|\phi(\bar{J}_1)-\phi(\bar{J}_2)\|=\Big\|\sum_{j\in J_2}x_j-\sum_{j\in J_1}x_j\Big\|&\ge x^*_k\Big(\sum_{j\in J_2}x_j-\sum_{j\in J_1}x_j\Big)\\
&=|J_2|\theta=\frac{\theta d_{\mathcal{T}}(\bar{J}_1,\bar{J}_2)}{2}.
\end{align*}

(ii)$\Rightarrow$(i): Suppose that $X$ has property ($\beta$), and there exist a pruned isometric subset $\mathcal{T}_k$ of $\mathcal{T}$, $k\in\mathbb{N}\cup\{0\}$, a Lipschitz map $\phi:\mathcal{T}_k\to X$ and a constant $\gamma>0$, such that (ii) holds. Without loss of generality we may assume that $\text{Lip}(\phi)=1$. Next we show that there is a sequence of subsets of $\mathcal{T}_k$, denoted by $\{\mathcal{T}_{k+n}\}_{n=0}^\infty$, that satisfies (1')-(3') in Section 2 and
$\text{Lip}(\left.\phi\right|_{\mathcal{T}_{k+n}})\le\tau^n$ for every $n\in\mathbb{N}\cup\{0\}$, where $\tau:=1-\frac{1}{2}\bar{\beta}_X(2\gamma)$.

The sequence $\{\mathcal{T}_{k+n}\}_{n=0}^\infty$ is defined in a similar way as that in Lemma \ref{subtree}.
Suppose $\mathcal{T}_{k+n}$ has been defined so that it has the above properties, we define $\mathcal{T}_{k+n+1}$ as follows: 
first choose the root $\emptyset$. Suppose $J\in\mathcal{T}_{k+n}$ has been chosen. Let $\{J\cup J_i\}_{i=1}^\infty$ be the immediate descendants of $J$ in $\mathcal{T}_{k+n}$, where $|J_i|=2^{k+n}$, $\max J<\min J_i$ and $\max J_i<\min J_{i+1}$ for all $i$. Consider the immediate descendants of $J\cup J_1$ in $\mathcal{T}_{k+n}$, denoted by $\{J\cup J_1\cup J_{1,i}\}_{i=1}^\infty$, where $|J_{1,i}|=2^{k+n}$, $\max J_1<\min J_{1,i}$ and $\max J_{1,i}<\min J_{1,i+1}$ for all $i$. Then we have
\begin{align*}
&\|\phi(J)-\phi(J\cup J_1)\|\le\tau^n 2^{k+n},\\
&\|\phi(J\cup J_1\cup J_{1,i})-\phi(J\cup J_1)\|\le\tau^n 2^{k+n},~i\in\mathbb{N}
\end{align*}

\vspace{-7mm}
\begin{align*}
\|\phi(J\cup J_1\cup J_{1,i})-\phi(J\cup J_1\cup J_{1,j})\|&\ge\gamma d_{\mathcal{T}}(J\cup J_1\cup J_{1,i},J\cup J_1\cup J_{1,j})\\
&=2^{k+n+1}\gamma,\hspace{4mm}i\neq j,
\end{align*}
Thus there exists $k_1\in\mathbb{N}$ such that
$$\|\phi(J)-\phi(J\cup J_1\cup J_{1,k_1})\|\le\tau^n 2^{k+n}\cdot2\Big(1-\frac{1}{2}\bar{\beta}_X\Big(\frac{2\gamma}{\tau^n}\Big)\Big)\le\tau^{n+1}2^{k+n+1},$$
where the second inequality follows from the fact that $\bar{\beta}_X(\cdot)$ is non-decreasing; denote $J\cup J_1\cup J_{1,k_1}:=J\cup\bar{J}_{k_1}$. Suppose $J\cup\bar{J}_{k_j}$ has been chosen, since $\max J_i<\min J_{i+1}$ for all $i$, there exists $k_{j+1}\in\mathbb{N}$ such that $\min J_{k_{j+1}}>\max \bar{J}_{k_j}$. Again we can choose an immediate descendant of $J\cup J_{k_{j+1}}$ in $\mathcal{T}_{k+n}$, denoted by $J\cup \bar{J}_{k_{j+1}}$ with $|\bar{J}_{k_{j+1}}|=2^{k+n+1}$ and $J_{k_{j+1}}<\bar{J}_{k_{j+1}}$, that satisfies
$$\|\phi(J)-\phi(J\cup \bar{J}_{k_{j+1}})\|\le\tau^{n+1}2^{k+n+1}.$$
Then the sequence $\{J\cup\bar{J}_{k_j}\}_{j=1}^\infty$ we have chosen consists of descendants of $J$ with height $|J|+2^{k+n+1}$, and for every $j\in\mathbb{N}$ one has $\max\bar{J}_{k_j}<\min\bar{J}_{k_j+1}$ and $$\|\phi(J)-\phi(J\cup \bar{J}_{k_j})\|\le\tau^{n+1}2^{k+n+1}.$$ The vertex set of $\mathcal{T}_{k+n+1}$ are those vertices chosen in this induction process.
Now it is easy to see that the sequence $\{\mathcal{T}_{k+n}\}_{n=0}^\infty$ has the desired properties, and it follows that $\gamma\le\tau^n$ for all $n\in\mathbb{N}\cup\{0\}$. Note that $\tau\in(0,1)$ since $X$ has property ($\beta$), we get a contradiction by letting $n\to\infty$.
\end{proof}

\begin{remark}\label{vmyold}
The original submetric characterization we have is the following:

\vspace{3mm}
\noindent (iii) There exist a pruned isometric subset $\mathcal{T}_k$ of $\mathcal{T}$, $k\in\mathbb{N}\cup\{0\}$, and a map $\phi:\mathcal{T}_k\to X$ that satisfies the following:

\vspace{1mm}
\begin{enumerate}[\hspace{5mm}(a)]
\item $\phi$ is an ancestor-to-descendant Lipschitz embedding, i.e., there exists $C,L>0$ such that for all $J<\bar{J}$ in $\mathcal{T}_k$, $$\frac{1}{C} d_{\mathcal{T}}(J,\bar{J})\le\|\phi(J)-\phi(\bar{J})\|\le Ld_{\mathcal{T}}(J,\bar{J}).$$

\item There exists $\gamma>0$ such that for every sequence $\{\mathcal{T}_{k+n}\}_{n=0}^\infty$ of subsets of $\mathcal{T}_k$ that satisfy (1')-(3') in section 2, and every $n\in\mathbb{N}\cup\{0\}$, $I,J\in\mathcal{T}_{k+n}$ with $I\neq J$ and $|I|=|J|$, $$\|\phi(I)-\phi(J)\|\ge 2^{k+n}\gamma.$$
\end{enumerate}

\vspace{1mm}
Shortly after we obtained the above submetric characterization (iii), we had a discussion with Florent Baudier. He made the observation that an improved version ((ii) in Theorem \ref{mcbeta}) can actually be derived from the work of Dilworth, Kutzarova, and Randrianarivony in \cite{DKR2016}. He also realized that the characterization can be proved using the notion of beta convexity, which is a metric analogue of property ($\beta$) introduced in his on-going work with Chris Gartland \cite{BG2021}.
\end{remark}

\section{Applications}

\subsection{Stability of property ($\beta$) under coarse Lipschitz embeddings}\ 

\vspace{1mm}
Our first application of the submetric characterization is the following stability result under coarse Lipschitz embeddings.

\begin{theorem}\label{clbeta}
Let $X$ be a Banach space that has an equivalent norm with property ($\beta$). If a Banach space $Y$ admits a coarse Lipschitz embedding into $X$, then $Y$ also has an equivalent norm with property ($\beta$).
\end{theorem}

\begin{proof}
Assume that $Y$ does not have an equivalent norm with property ($\beta$), then it follows from the proof of (i)$\Rightarrow$(ii) in Theorem \ref{mcbeta} that there exist $\gamma\in(0,1]$ and a map $\phi:\mathcal{T}\to Y$ that satisfy (a) and (b) as follows:
\begin{enumerate}[\hspace{4mm}(a)]
\item For all $J<\bar{J}$ in $\mathcal{T}$, $$\gamma d_{\mathcal{T}}(J,\bar{J})\le\|\phi(J)-\phi(\bar{J})\|\le d_{\mathcal{T}}(J,\bar{J}).$$
\item For every $J\in\mathcal{T}$, and every pair of descendants $\bar{J}_1=J\cup J_1$, $\bar{J}_2=J\cup J_2$ of $J$ that satisfies $|J_1|=|J_2|$ and $\max J_1<\min J_2$, $$\|\phi(\bar{J}_1)-\phi(\bar{J}_2)\|\ge\gamma d_{\mathcal{T}}(\bar{J}_1,\bar{J}_2).$$
\end{enumerate}
Denote $\phi(J):=v_J$ for $J\in\mathcal{T}$. 

Let $f:Y\to X$ be a coarse Lipschitz embedding, so there exists $d\ge0$ and $A,B>0$ such that for all $x,y\in Y$ with $\|x-y\|\ge d$,
$$\frac{1}{A}\|x-y\|\le\|f(x)-f(y)\|\le B\|x-y\|.$$
Choose $k\in\mathbb{N}$ such that $2^k\gamma>d$. Let $\mathcal{T}_k$ be a pruned isometric subset of $\mathcal{T}$. We claim that the map $g:=f\circ\left.\phi\right|_{\mathcal{T}_k}:\mathcal{T}_k\to X$, $g(J):=u_J$, $J\in\mathcal{T}_k$ satisfies
\begin{enumerate}
\item For all $J<\bar{J}$ in $\mathcal{T}_k$, $$\|g(J)-g(\bar{J})\|\le B d_{\mathcal{T}}(J,\bar{J}).$$
\item For every $J\in\mathcal{T}_k$, and every pair of descendants $\bar{J}_1=J\cup J_1$, $\bar{J}_2=J\cup J_2$ of $J$ in $\mathcal{T}_k$ that satisfies $|J_1|=|J_2|$ and $\max J_1<\min J_2$, $$\|g(\bar{J}_1)-g(\bar{J}_2)\|\ge\frac{\gamma d_{\mathcal{T}}(\bar{J}_1,\bar{J}_2)}{A}.$$
\end{enumerate}

\begin{proof}[Proof of the claim]\renewcommand{\qedsymbol}{}\ 

\vspace{1mm}
Note that for $J<\bar{J}$ in $\mathcal{T}_k$, we have $$\|v_J-v_{\bar{J}}\|\ge\gamma d_{\mathcal{T}}(J,\bar{J})\ge2^k\gamma>d,$$ and thus
$$\|u_J-u_{\bar{J}}\|\le B\|v_J-v_{\bar{J}}\|\le Bd_{\mathcal{T}}(J,\bar{J}),$$
which implies that $\text{Lip}(g)\le B$.

On the other hand, let $J\in\mathcal{T}_k$, and $\bar{J}_1=J\cup J_1$ and $\bar{J}_2=J\cup J_2$ are descendants of $J$ in $\mathcal{T}_k$ that satisfies $|J_1|=|J_2|$ and $\max J_1<\min J_2$, we have $$\|v_{\bar{J}_1}-v_{\bar{J}_2}\|\ge\gamma d_{\mathcal{T}}(\bar{J}_1,\bar{J}_2)\ge2^{k+1}\gamma>d,$$
which implies that 
$$\|u_{\bar{J}_1}-u_{\bar{J}_2}\|\ge\frac{1}{A}\|v_{\bar{J}_1}-v_{\bar{J}_2}\|\ge\frac{\gamma d_{\mathcal{T}}(\bar{J}_1,\bar{J}_2)}{A},$$
and this finishes the proof of the claim.
\end{proof}

Now it follows from Theorem \ref{mcbeta} (ii)$\Rightarrow$(i) that $X$ does not have an equivalent norm with property ($\beta$); this is a contradiction.
\end{proof}

\subsection{Stability of property ($\beta$) under nonlinear quotients}\ 

\vspace{1mm}
The notion of property ($\beta$) has been found very useful in the study of nonlinear quotients of Banach spaces. Lima and Randrianarivony \cite{LimaR2012} first brought in property ($\beta$) to study unifrom quotients of $\ell_p$ spaces. They proved that $\ell_q$ is not a uniform quotient of $\ell_p$ for $1<p<q$, and the main technique used is call the ``fork argument''. The ``fork'' here describes a configuration of points just as the $x$ and $\{x_n\}_{n=1}^\infty$ in the definition of property ($\beta$), and the ``fork argument" is a process of nonlinear lifting of points sitting approximately in such a position. Later, Baudier and Zhang \cite{BZ2016} proved the same result by estimating the $\ell_p$-distortion of the countably branching trees. Their proof shows that the nonexistence of nonlinear quoitent maps from $\ell_p$ to $\ell_q$ is actually due to the discrepancy between their modulus of property ($\beta$). Then it becomes a natural question whether the property ($\beta$) is preserved (up to renorming) under nonlinear quotients. Dilworth, Kutzarova and Randrianarivony \cite{DKR2016} gave an affirmative answer to the question in the uniform category. They applied the ``fork argument" to uniform quotient maps onto a diamond-type graph called parasol graph, and proved that if a separable Banach space $Y$ is a uniform quotient of a Banach space admitting an equivalent norm with property ($\beta$), then $Y$ also has an equivalent norm with property ($\beta$). Recently it was shown in \cite{DKLRbeta} that the separable condition can be removed.

The goal of this subsection is to provide a stability theorem of property ($\beta$) under nonlinear quotients in the coarse category. Our approach, which works as well in the context of uniform quotient, does not require the use of the parasol graph or the ``fork argument". Instead, we will apply the submetric characterization of property ($\beta$) and use the same technique as that in the proof of Theorem \ref{clbeta}.

We recall from \cite{BJLPS1999} that a map $f:X\to Y$ between metric spaces $X$ and $Y$ is called co-uniformly continuous if for every $d>0$, there exists $\delta=\delta(d)>0$ such that for all $x\in X$, $$B_Y(f(x),\delta)\subseteq f(B_X(x,d)).$$ If the $\delta$ satisfies $\delta\ge d/C$ for some constant $C>0$ independent of $d$, then the map $f$ is said to be co-Lipschitz.
A uniform (resp. Lipschitz) quotient map is a co-uniformly continuous (resp. co-Lipschitz) map that is also uniformly continuous (resp. Lipschitz), and we say that $Y$ is a uniform (resp. Lipschitz) quotient of $X$ if there exists a uniform (resp. Lipschitz) quotient map from $X$ to $Y$.

The notion of coarse quotient was introduced by Zhang in \cite{Zhang2015}. A map $f:X\to Y$ between metric spaces $X$ and $Y$ is called co-coarsely continuous if there is a constant $K\ge0$ that satisfies the following: for every $d>K$, there exists $\delta=\delta(d)>0$ such that for all $x\in X$, $$B_Y(f(x),d)\subseteq f(B_X(x,\delta))^K.$$
Here for a set $A$ in a metric space $X$, $A^K:=\bigcup_{a\in A}B_X(a,K)$.
A coarse quotient map is a co-coarsely continuous map that is also coarsely continuous, and we say that $Y$ is a coarse quotient of $X$ if there exists a coarse quotient map from $X$ to $Y$.

Co-Lipschitz maps are surjective. A co-uniformly continuous map $f:X\to Y$ is surjective if the target space $Y$ is connected.
For a co-coarsely continuous map $f:X\to Y$, we only have $Y=f(X)^K$ for some $K\ge0$. Therefore, if the target space of a coarse quotient map $f$ is unbounded, one must have $\text{Lip}_d(f)>0$ for all $d>0$.

The lemma below says that co-uniformly continuous maps and co-coarsely continuous maps are ``co-Lipschtiz for large distances'' provided the target space is metrically convex. We refer to \cite{BJLPS1999,Zhang2018} for the proof.  

\begin{lemma}\label{coLiplarge}
Let $f:X\to Y$ be a map from a metric space $X$ to a metrcially convex metric space $Y$.
\begin{enumerate}
\item If $f$ is co-uniformly continuous, then for every $d>0$, there exists $C=C(d)>0$ such that for all $x\in X$ and $r\ge d$, $$B_Y(f(x),r/C)\subseteq f(B_X(x,r)).$$
\item If $f$ is co-coarsely continuous, then there exists a constant $K\ge0$ that satisfies the following: for every $d>K$, there exists $C=C(d)>0$ such that for all $x\in X$ and $r\ge d$, $$B_Y(f(x),r)\subseteq f(B_X(x,Cr))^K.$$
\end{enumerate}
\end{lemma}

\begin{corollary}\label{nq}
Let $X$ and $Y$ be two metric spaces and $f:X\to Y$ be a map that is co-uniformly continuous or co-coarsely continuous. Assume that $Y$ is metrically convex, then there exist $K>0$ and $C>0$ such that for all $x\in X$ and $r>0$, $$B_Y(f(x),r)\subseteq f(B_X(x,Cr))^K.$$
\end{corollary}

\begin{proof}
If $f$ is co-uniformly continuous, we fix $d>0$, then it follows from Lemma \ref{coLiplarge} (1) that there exists $C=C(d)>0$ such that for all $x\in X$ and $r\ge d$, $$B_Y(f(x),r/C)\subseteq f(B_X(x,r))\subseteq f(B_X(x,r))^{d/C}.$$ Thus for all $x\in X$ and $r>0$, $$B_Y(f(x),r/C)\subseteq f(B_X(x,r))^{d/C},$$ and the result follows.
	
Similarly, if $f$ is co-coarsely continuous, note that the constant $C$ in Lemma \ref{coLiplarge} (2) satisfies $B_Y(f(x),r)\subseteq f(B_X(x,Cr))^d$ for all $x\in X$ and $r>0$, again the result follows.
\end{proof}

\begin{theorem}\label{nqbeta}
Let $X$ be a Banach space that has an equivalent norm with property ($\beta$). If a Banach space $Y$ is a uniform or coarse quotient of a subset of $X$, where the quotient map is Lipschitz for large distances, then $Y$ also has an equivalent norm with property ($\beta$).
\end{theorem}

\begin{proof}
Assume that $Y$ does not have an equivalent norm with property ($\beta$), then it follows from the proof of (i)$\Rightarrow$(ii) in Theorem \ref{mcbeta} that there exist $\gamma\in(0,1]$ and a map $\phi:\mathcal{T}\to Y$ that satisfy (a) and (b) as follows:
\begin{enumerate}[\hspace{4mm}(a)]
\item For all $J<\bar{J}$ in $\mathcal{T}$, $$\gamma d_{\mathcal{T}}(J,\bar{J})\le\|\phi(J)-\phi(\bar{J})\|\le d_{\mathcal{T}}(J,\bar{J}).$$
\item For every $J\in\mathcal{T}$, and every pair of descendants $\bar{J}_1=J\cup J_1$, $\bar{J}_2=J\cup J_2$ of $J$ that satisfies $|J_1|=|J_2|$ and $\max J_1<\min J_2$, $$\|\phi(\bar{J}_1)-\phi(\bar{J}_2)\|\ge\gamma d_{\mathcal{T}}(\bar{J}_1,\bar{J}_2).$$
\end{enumerate}
Denote $\phi(J):=v_J$ for $J\in\mathcal{T}$. 

Let $S$ be a subset of $X$ and $f:S\to Y$ be a uniform or coarse quotient map that is Lipschitz for large distances. By Corollary \ref{nq}, there exist $K>0$ and $C>0$ such that for all $x\in S$ and $r>0$, 
\begin{align}
B_Y(f(x),r)\subseteq f(B_S(x,Cr))^K.\label{nqcoLip}
\end{align}
Let $d>0$ satisfy $\text{Lip}_d(f)\in (0,\infty)$ and $\omega_f(d)<\infty$. Choose $k\in\mathbb{N}$ such that $2^k\gamma>\omega_f(d)+2K$. Let $\mathcal{T}_k$ be a pruned isometric subset of $\mathcal{T}$. We define a map $g:\mathcal{T}_k\to X$ as follows: first choose any $u_\emptyset\in S$ such that $\|v_\emptyset-f(u_\emptyset)\|\le K$. Suppose $u_J\in S$ has been defined for $J\in\mathcal{T}_k$ such that $\|v_J-f(u_J)\|\le K$. Let $\bar{J}$ be an immediate descendant of $J$ in $\mathcal{T}_k$. It follows from \eqref{nqcoLip} that $$v_{\bar{J}}\in B_Y\big(f(u_J),\|v_{\bar{J}}-v_J\|+K\big)\subseteq f\big(B_S\big(u_J,C\big(\|v_{\bar{J}}-v_J\|+K\big)\big)\big)^K,$$
so there exists $u_{\bar{J}}\in S$ such that $\|u_{\bar{J}}-u_J\|\le C\big(\|v_{\bar{J}}-v_J\|+K\big)$ and $\|v_{\bar{J}}-f(u_{\bar{J}})\|\le K$. This induction process defines a $u_J\in S$ for every $J\in\mathcal{T}_k$; thus $g:\mathcal{T}_k\to X$, $g(J):=u_J$, $J\in\mathcal{T}_k$ is a well-defined map. We claim that the map $g$ satisfies
\begin{enumerate}[(i)]
\item For all $J<\bar{J}$ in $\mathcal{T}_k$, $$\|g(J)-g(\bar{J})\|\le 2C d_{\mathcal{T}}(J,\bar{J}).$$
\item For every $J\in\mathcal{T}_k$, and every pair of descendants $\bar{J}_1=J\cup J_1$, $\bar{J}_2=J\cup J_2$ of $J$ in $\mathcal{T}_k$ that satisfies $|J_1|=|J_2|$ and $\max J_1<\min J_2$, $$\|g(\bar{J}_1)-g(\bar{J}_2)\|\ge\frac{\gamma d_{\mathcal{T}}(\bar{J}_1,\bar{J}_2)}{2\text{Lip}_d(f)}.$$
\end{enumerate}
\begin{proof}[Proof of the claim]\renewcommand{\qedsymbol}{}\

\vspace{1mm}
If $\bar{J}$ is an immediate descendant of $J$ in $\mathcal{T}_k$, we have
$$\|u_J-u_{\bar{J}}\|\le C\big(\|v_J-v_{\bar{J}}\|+K\big)\le C(2^k+K)<2^{k+1}C=2C d_{\mathcal{T}}(J,\bar{J}),$$ which implies that $g$ is Lipschitz with Lipschitz constant at most $2C$.

On the other hand, let $J\in\mathcal{T}_k$, and $\bar{J}_1=J\cup J_1$ and $\bar{J}_2=J\cup J_2$ are descendants of $J$ in $\mathcal{T}_k$ that satisfies $|J_1|=|J_2|$ and $\max J_1<\min J_2$, we must have $\|u_{\bar{J}_1}-u_{\bar{J}_2}\|\ge d$, otherwise $$2^{k+1}\gamma\le\gamma d_{\mathcal{T}}(\bar{J}_1,\bar{J}_2)\le\|v_{\bar{J}_1}-v_{\bar{J}_2}\|\le\|f(u_{\bar{J}_1})-f(u_{\bar{J}_2})\|+2K\le\omega_f(d)+2K,$$ which contradicts the choice of $k$. Therefore, 
\begin{align*}
\|u_{\bar{J}_1}-u_{\bar{J}_2}\|&\ge\frac{1}{\text{Lip}_d(f)}\|f(u_{\bar{J}_1})-f(u_{\bar{J}_2})\|\\
&\ge\frac{1}{\text{Lip}_d(f)}\big(\|v_{\bar{J}_1}-v_{\bar{J}_2}\|-2K\big)\\
&\ge\frac{1}{\text{Lip}_d(f)}(\gamma d_{\mathcal{T}}(\bar{J}_1,\bar{J}_2)-2K)\ge\frac{\gamma d_{\mathcal{T}}(\bar{J}_1,\bar{J}_2)}{2\text{Lip}_d(f)},
\end{align*}
and this finishes the proof of the claim.
\end{proof}

Now it follows from Theorem \ref{mcbeta} (ii)$\Rightarrow$(i) that $X$ does not have an equivalent norm with property ($\beta$); this is a contradiction.
\end{proof}

\begin{remark}
In Theorem \ref{nqbeta}, if $Y$ is a coarse quotient of a subset of $X$, then it is enough to assume that the coarse quotient map is coarse Lipschitz rather than Lipschitz for large distances.
\end{remark}

\begin{corollary}
Let $X$ be a Banach space admitting an equivalent norm with property ($\beta$). If a Banach space $Y$ is (i) a Lipschitz quotient of a subset of $X$; or (ii) a uniform quotient of $X$; or (iii) a coarse quotient of $X$, then $Y$ also has an equivalent norm with property ($\beta$).
\end{corollary}

\noindent
{\bf Acknowledgement.} The author would like thank Florent Baudier and Chris Gartland for very helpful discussion which improved Theorem \ref{mcbeta}. See Remark \ref{vmyold}.

\bibliographystyle{amsplain}

\end{document}